\documentclass{article}[11pt]
\usepackage[utf8]{inputenc}
\usepackage[paper=a4paper,dvips,top=3cm,left=2.4cm,right=2.4cm,foot=1cm,bottom=3cm]{geometry}
\usepackage{amssymb}
\usepackage{amsmath}
\usepackage{authblk}
\usepackage{amsthm}
\usepackage{graphicx}
\usepackage{subfigure}
\usepackage{epstopdf}
\usepackage{color}
\usepackage{soul}
\usepackage{float}
\usepackage{booktabs}
\usepackage{multirow}
\usepackage{adjustbox}
\usepackage[graphicx]{realboxes}
\usepackage{rotating}
\usepackage{array}
\usepackage{makecell}

\usepackage{footmisc}
\usepackage{caption}

\usepackage{url}

\newtheorem{theorem}{Theorem}[section]
\newtheorem{definition}{Definition}[section]
\newtheorem{lemma}{Lemma}[section]

\usepackage{algorithm}
\usepackage{algorithmic}

\begin{document}

\title{\huge Reweighted Quasi Norm Regularized Low-Rank Factorization for Matrix Robust PCA}
\author[]{Zhenzhi Qin\protect\footnotemark[1],\quad Liping Zhang\protect\footnotemark[1]}

\maketitle

\date{}

\renewcommand{\thefootnote}{\fnsymbol{footnote}}
\footnotetext[1]{Department of Mathematical Sciences, Tsinghua University, Beijing, China ({\tt qzz19@mails.tsinghua.edu.cn}; \quad {\tt lipingzhang@mail.tsinghua.edu.cn}).}
\renewcommand{\thefootnote}{\arabic{footnote}}

\vspace{-0.2in}

\begin{abstract} 

Robust Principal Component Analysis (RPCA) and its associated non-convex relaxation methods constitute a significant component of matrix completion problems, wherein matrix factorization strategies effectively reduce dimensionality and enhance computational speed. However, some non-convex factorization forms lack theoretical guarantees. This paper proposes a novel strategy in non-convex quasi-norm representation, introducing a method to obtain weighted matrix quasi-norm factorization forms. Especially, explicit bilinear factor matrix factorization formulations for the weighted logarithmic norm and weighted Schatten-$q$ quasi norms with $q=1, 1/2, 2/3$ are provided, along with the establishment of corresponding matrix completion models. An Alternating Direction Method of Multipliers (ADMM) framework algorithm is employed for solving, and convergence results of the algorithm are presented.

{\bf keywords}: Robust principal component analysis; Schatten-q quasi norm; Logarithmic quasi norm; bilinear factor matrix factorization
\end{abstract}

\section{Introduction}

In the realm of modern data analysis and signal processing, techniques for data decomposition and regularization play a pivotal role. These methodologies not only furnish indispensable tools for unraveling the underlying structure of data but also find wide-ranging applications across various domains, including recommendation systems \cite{gantner2010factorization}, image processing \cite{hsieh2014nuclear,aharon2006k}, medical image analysis \cite{selvan2020tensor}, intelligent transportation \cite{guo2022traffic}, and robust principal component analysis (RPCA) \cite{candes2011robust}. RPCA, also referred to as low-rank and sparse matrix decomposition in \cite{tao2011recovering}, \cite{zhou2011godec}, or robust matrix completion in \cite{chen2011robust}, aims to recover a low-rank matrix $X\in\mathbb{R}^{n\times m}$ and a sparse matrix $S\in\mathbb{R}^{n\times m}$ from corrupted observations $M=X^+S^\in\mathbb{R}^{n\times m}$ as follows:
\begin{equation}\label{eq:model-rank}
    \min_{X}\lambda{\rm rank}(X)+\|S\|_0,\quad {\rm s.t.}\ X+S-M=0,
\end{equation}
where $|\cdot|_0$ denotes the $\ell_0$-norm, and $\lambda>0$ represents a regularization parameter. However, solving \eqref{eq:model-rank} is NP-hard. Consequently, an effective approach to tackle this challenge is to relax \eqref{eq:model-rank} into a convex problem. To derive a more generalized model, \cite{chen2011robust} extends \eqref{eq:model-rank} into the following framework:
\begin{equation}\label{eq:RPCA-model}
    \min_{X}\lambda\|X\|_{*}+\|S\|_1,\ {\rm s.t.}\quad P_{\Omega}(X+S-M)=0.
\end{equation}
Here, $P_{\Omega}$ represents the projection onto the index set $\Omega$, and $M$ denotes the observed matrix. Wright et al. \cite{wright2009robust} and Candès et al. \cite{candes2011robust} have shown that under certain mild conditions, the convex relaxation formulation \eqref{eq:RPCA-model} can effectively recover the low-rank and sparse matrices $X^*$ and $S^*$ with high probability.

In pursuit of improved recovery efficacy, researchers have explored various non-convex representation techniques to bring the model closer to low-rank solutions. Apart from the classical nuclear norm relaxation, numerous recent studies have put forth effective and computationally tractable non-convex relaxation methods\cite{lai2013improved,lu2014smoothed,nie2012robust,shang2017bilinear,hu2012fast,chen2021logarithmic,gu2017weighted,recht2010guaranteed,kang2016top}. \cite{nie2012robust} proposes a robust matrix completing model based on Schatten-q norm ($q<1$), where Schatten-q norm is defined as follows:
\begin{equation*}
	\|X\|_{S_q}=\left(\sum_{i=1}^{\min\{ m,n\}}\sigma_{i}^q\right)^{1/q},
\end{equation*}
where $\sigma_i$ denotes the $i$-th singular value of the matrix $X$. In comparison with the nuclear norm, the Schatten-$q$ quasi-norm with $0 < q < 1$ is closer to approximating the rank function \cite{lai2013improved,lu2014smoothed}. The logarithmic determinant ($\log\det$) function yields a smaller value than the nuclear norm, thus employing it to approximate the rank function results in enhanced low-rank performance \cite{kang2016top}. Essentially, the logarithmic determinant involves applying the logarithm function to each singular value and summing them up:
\begin{equation*}
    \|X\|_L^p=\sum_{i=1}^{\min\{ m,n\}}\log(\sigma_i^p+\epsilon).
\end{equation*}
Kang et al. \cite{kang2016top} leverage this property and devise an Alternating Direction Method of Multipliers (ADMM) algorithm to address recovery problems. Re-weighting strategies have been widely employed in matrix contexts \cite{gu2017weighted}, this strategy enhances the effectiveness of the shrinkage operator.

However, the aforementioned models involve non-convex relaxation methods that operate directly on singular value vectors, necessitating matrix singular value decomposition, which incurs significant computational costs. Despite numerous attempts at fast SVD computation, such as partial SVD \cite{larsen2005propack}, the performance of these methods still falls short for many practical applications \cite{oh2015fast}. It has been observed that representing relaxation functions in the form of matrix factorizations can effectively reduce the dimensionality of data and offer computational advantages \cite{shang2017bilinear,chen2021logarithmic,recht2010guaranteed}. \cite{recht2010guaranteed} propose a factorization form of the nuclear norm, termed the BM equation:
\begin{equation}\label{eq:BM}
		\|X\|_{*}=\min_{A\in\mathbb{R}^{m\times d},B\in\mathbb{R}^{n\times d},X=AB^T}\frac{1}{2}\left(\|A\|_{F}^2+\|B\|_{F}^2\right).
\end{equation}
Replacing the nuclear norm in \eqref{eq:RPCA-model}, \cite{cabral2013unifying} propose a corresponding model and employ an ADMM method to solve it. Furthermore, Shang et al. \cite{shang2017bilinear} extend the concept of matrix factorization to Schatten-$q$ norms, offering matrix factorization forms for $q=0.5$ and $q=2/3$. They propose corresponding factorization models based on this theory and validate the recovery performance. Additionally, Chen et al. \cite{chen2021logarithmic} present a factorization form of the logarithmic determinant, significantly enhancing computational efficiency.

Chen et al. \cite{chen2022reweighted} introduced a reweighted nuclear norm factorization form. However, in the factorization form proposed by \cite{chen2022reweighted}, the weights are not freely selectable; rather, they must be determined based on the singular values of the raw matrix, which is also referred to as the ``deep prior." In practical applications of matrix recovery, fulfilling the ``deep prior" condition is nearly impossible. Hence, exploring a weighted matrix factorization form where the weights $W$ are independent of the matrix to be recovered represents an area for further research.

Moreover, integrating the re-weighted strategy with non-convex representations such as Schatten-q norms or Logarithmic norms currently lacks theoretical models and effective algorithms. Consequently, designing the factorization form of re-weighted quasi-norms emerges as a research topic worthy of exploration, which not only extends the theory of matrix recovery problems but also holds promise for diverse application prospects.

We summarize the main contributions of this work as follows:
\begin{itemize}
    \item We investigate the relationship between the weight matrix and the singular value vector, and provide an analytical strategy for matrix-weighted quasi norms.
    \item For Schatten-q norms ($q=1, 0.5$, and $2/3$) and Logarithmic norms, we demonstrate the existence of their factorization forms and provide specific formulations.
    \item We establish the matrix completion models for Schatten-q norms ($q=1, 0.5$, and $2/3$) and Logarithmic norms, utilizing an ADMM framework algorithm for their solution. Additionally, we provide convergence results for this algorithm.
\end{itemize}

The remaining structure of this article is as follows. Section \ref{sec:re-1} presents the relevant knowledge of manifold optimization and matrix theory required for the study. Section \ref{sec:re-2} delves into the action of the weight matrix on the original singular value matrix in the strategy of weighted quasi-nuclear norms and establishes relevant inequalities, extending the discussion to complex matrices and quaternion matrices. Section \ref{sec:re-3} investigates the factorization forms of weighted Schatten-q norms with $q=1, 0.5, 2/3$, and the Logarithmic norm, providing the corresponding factorization formulations. Section \ref{sec:re-4} establishes the corresponding low-rank matrix recovery models and solution methods for $q=1$, $q=0.5$, and $q=2/3$ in Schatten-q norm and Logarithmic norm. It offers detailed analytical solutions to each subproblem and provides convergence analysis to the Karush-Kuhn-Tucker (KKT) point of the problem under mild conditions.

\section{Preliminaries}\label{sec:re-1}

For proposing the factorization strategy for weighted quasi-norms, we rely on certain results from manifold optimization and matrix analysis. Initially, we revisit some foundational concepts of manifold optimization, as outlined in \cite{boumal2023introduction}.

\begin{definition}\label{def:embedmanifold}
	Let $\mathcal{M}$ be a subset of a linear space $\mathcal{E}$. We say $\mathcal{M}$ is an embedded submanifold if either of the following holds:
	\begin{itemize}
		\item[1.] $\mathcal{M}$ is an open subset of $\mathcal{E}$, then, we call $\mathcal{M}$ an open submanifold. If $\mathcal{M}=\mathcal{E}$, we also call it a linear manifold.
  
		\item[2.]  For a fixed integer $k\geq1$ and for each $x\in\mathcal{M}$ there exists  a neighborhood $U$ of $x$ in $\mathcal{E}$ and a smooth function $h:\ U\rightarrow\mathbb{R}^k$ such that
		\begin{itemize}
			\item [\rm (a)] if $y\in U$,  then $y\in\mathcal{M}$ if and only if $h(y)=0$;
			\item [\rm (b)] ${\rm rank}\ {\rm D}h(x)=k$, where ${\rm D}h$ represents the derivative of $h$.
		\end{itemize}
	\end{itemize}
\end{definition}
Such a function $h$ in Definition \ref{def:embedmanifold} is called a local defining function for $\mathcal{M}$ at $x$. This characterization implies that the manifold is locally homeomorphic to a $k$-dimensional Euclidean space. In our study, we primarily concentrate on the Stiefel manifold.
\begin{definition}
	The set of all column-orthogonal matrices of the same size is referred to as the Stiefel manifold, denoted as
	\begin{equation*}
		{\rm St}(n,p)=\{X\in\mathbb{R}^{n\times p}:\ X^TX=I_p\}.
	\end{equation*}
	Its corresponding local defining function is $h(X)=X^TX-I_p$.
\end{definition}
Clearly, ${\rm St}(n,p)$ is an embedded submanifold owing to the existence of the local defining function $h$. Tangent spaces can be utilized to elucidate the local properties of the manifold.
\begin{definition}
	Let $\mathcal{M}$ be a subset of a linear space $\mathcal{E}$. For all $x\in\mathcal{M}$, define 
	\begin{equation*}
		T_x\mathcal{M}=\{c'(0):\ c:\ [0,1]\rightarrow\mathcal{M}\text{ is smooth around }0\text{ and }c(0)=x\}.
	\end{equation*}
        We call $T_x\mathcal{M}$ the tangent space to $\mathcal{M}$ at $x$. Vectors in $T_x\mathcal{M}$ are called tangent vectors to $\mathcal{M}$ at $x$.
\end{definition}
The definition of the tangent space above may not be clear and intuitive. \cite{boumal2023introduction} provides a clear expression for $T_X{\rm St}(n,p)$ as follows:
\begin{equation}\label{eq:manifold-st-tanget}
	T_X{\rm St}(n,p)=\{V\in\mathbb{R}^{n\times p}:\ X^TV+V^TX=0\}.
\end{equation}

\begin{definition}[Riemann manifold]
	An inner product on $T_x\mathcal{M}$ is a bilinear, symmetric, and positive definite function $\langle\cdot,\cdot\rangle_x:\ T_x\mathcal{M}\times T_x\mathcal{M}\rightarrow\mathbb{R}$. Such an inner product $\langle\cdot,\cdot\rangle_x$ is called a smooth variation on $\mathcal{M}$ with $x$ if the function $x\rightarrow\langle V(x),W(x)\rangle_x$ is  smooth from $\mathcal{M}$ to $\mathbb{R}$ for any two smooth vector fields $V, W$ on $\mathcal{M}$. A manifold with smooth variations is called a Riemannian manifold.
\end{definition}
As described in \cite[Chapter 7]{boumal2023introduction}, ${\rm St}(n,p)$ is a Riemannian manifold. Optimal conditions for nonsmooth optimization on Riemannian manifolds were discussed in \cite{Zhang2013Optimality}. Firstly, we recall some definitions of the subdifferential in Riemannian manifolds.
\begin{definition}[Chart \cite{boumal2023introduction}]
	Let $\mathcal{M}$ be a manifold. If $U\subset\mathcal{M}$ and a smooth function $\phi:\ U\rightarrow\mathbb{R}^d$ satistying:
    \begin{itemize}
        \item $\phi(U)$ is an open set in $\mathbb{R}^d$,
        \item $\phi$ is invertible between $U$ and $\phi(U)$,
    \end{itemize}
    then we call $(U,\phi)$ is a chart of $\mathcal{M}$.
\end{definition}
According to \cite[Chapter 8]{boumal2023introduction}, we know that every point $x$ in the embedded manifold $\mathcal{M}$ possesses a chart $(U,\phi)$ such that $x\in U$. The concept of Lipschitz continuity on a manifold is defined as follows, as per \cite[Lemma 3.1]{Zhang2013Optimality}:
\begin{definition}[Lipschitz Continuity on Manifolds \cite{Zhang2013Optimality}]
	If for any $x \in \mathcal{M}\subset\mathbb{R}^m$, function $F:\mathbb{R}^m\rightarrow\mathbb{R}$ satisfies that $F\circ\phi^{-1}$ is Lipschitz continuous in the $\phi(V)$, where $(V,\phi)$ is a chart of $\mathcal{M}$ which contains $x$, then the function $F$ is locally Lipschitz continuous on the manifold $\mathcal{M}$.
\end{definition}
Similar to optimization analysis in Euclidean spaces, with Lipschitz continuity, we can introduce generalized Clarke subdifferentials.
\begin{definition}[Generalized Clarke Subdifferential \cite{Hpsseini2011Generalized}]
	For a locally Lipschitz function $F$ on $\mathcal{M}$, the Riemannian generalized directional derivative at point $x\in\mathcal{M}$ in direction $v$ is defined as
	\begin{equation*}
		F^{\circ}(x, v ) = \mathop{\lim\sup}_{y \rightarrow x,t\downarrow 0}\frac{F \circ \phi^{-1}(\phi (y ) + tD\phi (x)[v ]) - F \circ \phi^{-1}(\phi (y ))}{t},
	\end{equation*}
	where $(U,\phi)$ is the chart at $x$ of manifold $\mathcal{M}$. The generalized Clarke subdifferential of $F$ at $x\in \mathcal{M}$ is denoted as $\partial F(x)$ and is defined as
	\begin{equation*}
		\partial F(x) = \{ \xi \in T_x\mathcal{M} : \langle \xi , v \rangle \leq F^{\circ}(x, v ),\ \forall v \in T_x\mathcal{M} \} .
	\end{equation*}
\end{definition}
\begin{definition}[Definition 5.2 \cite{Zhang2013Optimality}]
	For a locally Lipschitz function $F$, if 
	\begin{itemize}
		\item the limit $F'(x; v )\doteq\lim_{t\downarrow 0}\frac{F(x+tv ) - F(x)}{t}$ exists for all $v \in T_x\mathcal{M}$, and
		\item $F'(x; v ) = F^{\circ }(x; v )$ for all $v \in T_x\mathcal{M}$,
	\end{itemize}
	then $F$ is regular along $T_x\mathcal{M}$ at $x \in \mathcal{M}$.
\end{definition}

In fact, when the function $F$ is regular along $T_x\mathcal{M}$ at $x \in \mathcal{M}$, the generalized Clarke subdifferential of $F$ on $\mathcal{M}$ can be expressed concisely as:
\begin{lemma}[Theorem 5.1 (iii) \cite{Zhang2013Optimality}]\label{lemma: manifold1}
	Let $\mathcal{M}$ be an embedded submanifold of $\mathbb{R}^{m}$. $F$ is Lipschitz continuous at $x\in\mathcal{M}$, and $\bar{F}\doteq F|_{\mathcal{M}}$. If $F$ is regular along $T_x\mathcal{M}$ at $x \in \mathcal{M}$, then we have $\partial\bar{F}(x)={\rm Proj}_{T_x\mathcal{M}}(\partial F(x))$, where ${\rm Proj}$ is the projection operator.
\end{lemma}
\begin{lemma}[Lemma 5.1  \cite{Zhang2013Optimality}]\label{lemma: manifold2}
	Let $F = F_1 + F_2$ be a function on $\mathbb{R}^{m}$, where $F_1$ is convex and $F_2$ is continuously differentiable. Then $F$ is regular along $T_x\mathcal{M}$ for any $x\in \mathcal{M}$.
\end{lemma}
For unconstrained optimization problems on manifolds $\mathcal{M}$, we can derive first-order necessary conditions based on \cite[Theorem 4.1]{Zhang2013Optimality}.
\begin{lemma}[Theorem 4.1 \cite{Zhang2013Optimality}]\label{lemma: manifold3}
	If $x^{*}$ is a local minimum solution to the problem $\min_{x\in\mathcal{M}} F(x)$, then $0\in \partial F(x^{*})$.
\end{lemma}

To derive the factorization strategy for the weighted quasi-norm, we require some analytical tools concerning the singular value vectors. Let's introduce some notation regarding singular values. For a matrix $X\in\mathbb{R}^{m\times n}$, its singular value decomposition (SVD) is given by $X=U\Sigma V^T$, where $\Sigma={\rm diag}(\sigma)$ and $\sigma\in\mathbb{R}^{\min{m,n}}$ represents the singular value vector. To simplify notation, we use $\sigma(X)$ to denote the singular value vector function of the matrix $X$.

Lewis \cite{lewis1995convex} provided a differential expression for compositions of singular value vectors and absolutely symmetric.
\begin{definition}
	A generalized function $p:\ \mathbb{R}^n\rightarrow [-\infty,+\infty]$ is called absolutely symmetric if $p(\gamma)=p(\sigma({\rm diag}(\gamma)))$ holds for any vector $\gamma\in\mathbb{R}^n$.
\end{definition}
\begin{lemma}[Corollary 2.5 \cite{lewis1995convex}]\label{lemma:diff-of-sigma}
	If the function $p$ is absolutely symmetric, then
	\begin{equation*}
		\partial (p\circ\sigma)(X)=\{U{\rm diag}(\mu)V^T:\ \mu\in\partial f(\sigma(X)),\ X=U{\rm diag}(\sigma(X))V^T\text{ is the SVD of $X$}\}.
	\end{equation*}
\end{lemma}

Shang et al. \cite{shang2017bilinear} provided the following factorization forms for Schatten-q norms.
\begin{lemma}\label{lemma:schatten-12-23}
	For a matrix $X\in\mathbb{R}^{m\times n}$ with ${\rm rank}(X)=r\leq d$, the factorization form of the Schatten-0.5 norm is:
	\begin{equation*}
		\|X\|_{S_{\frac{1}{2}}}=\min_{A\in\mathbb{R}^{m\times d},B\in\mathbb{R}^{n\times d},X=AB^T}\frac{1}{4}\left(\|A\|_{*}+\|B\|_{*}\right)^2.
	\end{equation*}
	The factorization form of the Schatten-$2/3$ norm is:
	\begin{equation*}
		\|X\|_{S_{\frac{2}{3}}}=\min_{A\in\mathbb{R}^{m\times d},B\in\mathbb{R}^{n\times d},X=AB^T}\left(\frac{\|A\|_{F}^2+2\|B\|_{*}}{3}\right)^{\frac{3}{2}}.
	\end{equation*}
\end{lemma}
Another non-convex function that possesses a bilinear factor matrix formulation is the Logarithmic norm \cite{chen2021logarithmic}.
\begin{lemma}\label{lemma:logdet-BM}
	For a matrix $X\in\mathbb{R}^{m\times n}$ with ${\rm rank}(X)=r\leq d$, the factorization form of the Logarithmic Norm is:
	\begin{equation*}
		2\|X\|_{L}^{1/2}=\min_{A\in\mathbb{R}^{m\times d},B\in\mathbb{R}^{n\times d},X=AB^T}\|A\|_{L}^1+\|B\|_{L}^1.
	\end{equation*}
\end{lemma}
\section{Factorization Strategy }\label{sec:re-2}

In this section, we will establish the factorization strategy for the weighted Schatten-$q$ norm and the weighted logarithmic norm. We begin by providing the definition of the block diagonal matrix form.

\begin{definition}
	A matrix $X\in\mathbb{R}^{n\times n}$ is said to have a block diagonal matrix form if it can be written as
	\begin{equation*}
		X=\begin{bmatrix}
			X_1 & 0 & \dots & 0\\
			0 & X_2 & \dots & 0\\
			0 & 0 & \ddots & 0\\
			0 & 0 & \dots & X_r
		\end{bmatrix},
	\end{equation*}
	where $X_i\in \mathbb{R}^{n_i\times n_i}, 1\le i\leq r$ with $\sum^r_{i=1}n_i=n$. We say that  such a matrix $X$ has the form $(n_1,\dots,n_r)$.
\end{definition}

The following lemma establishes the connection between symmetric matrices and diagonal matrices.

\begin{lemma}\label{lemma:sym-and-diag}
	Let $A\in\mathbb{R}^{n\times n}$ be a symmetric matrix and  $S$ be  a diagonal matrix in  the form of 
	\begin{equation*}
		S=\begin{bmatrix}
			s_1I_1 & 0 & \dots & 0\\
			0 & s_2I_2 & \dots & 0\\
			0 & 0 & \ddots & 0\\
			0 & 0 & \dots & s_rI_r
		\end{bmatrix},
	\end{equation*}
	where $s_1,\dots,s_r$ are nonzero distinct real numbers, $I_i$ is an  $n_i\times n_i$ identity matrix for $i=1,\ldots,r$ with $\sum^r_{i=1}n_i=n$. If $AS$ is symmetric, then $A$ and $S$ have the same form $(n_1,\dots,n_r)$.
\end{lemma}

\begin{proof}
	Let
	\begin{equation*}
		A=\begin{bmatrix}
			A_{11} & A_{12} & \dots & A_{1r}\\
			A_{21} & A_{22} & \dots & A_{2r}\\
			\dots & \dots & \ddots & \dots\\
			A_{r1} & A_{r2} & \dots & A_{rr}
		\end{bmatrix},
	\end{equation*}
	where $A_{ij}\in\mathbb{R}^{n_i\times n_j}$. Since $AS$, $A$, and $S$ are all symmetric, we have $AS=SA$, i.e.,
	\begin{equation*}
		\begin{bmatrix}
			s_1A_{11} & s_2A_{12} & \dots & s_rA_{1r}\\
			s_1A_{21} & s_2A_{22} & \dots & s_rA_{2r}\\
			\dots & \dots & \ddots & \dots\\
			s_1A_{r1} & s_2A_{r2} & \dots & s_rA_{rr}
		\end{bmatrix}=AS=SA=\begin{bmatrix}
			s_1A_{11} & s_1A_{12} & \dots & s_1A_{1r}\\
			s_2A_{21} & s_2A_{22} & \dots & s_2A_{2r}\\
			\dots & \dots & \ddots & \dots\\
			s_rA_{r1} & s_rA_{r2} & \dots & s_rA_{rr}
		\end{bmatrix},
	\end{equation*}
	Hence, for $i\ne j$, it holds $s_iA_{ij}=s_jA_{ij}$, which together with $s_i\ne s_j\ne 0$, implies $A_{ij}=0$. So, $A$ has the form $(n_1,\dots,n_r)$.
\end{proof}

Let $k$ be a positive integer. The following lemma demonstrates the uniqueness of the $1/k$th power form of symmetric positive definite matrices. 

\begin{lemma}\label{lemma:unique-1k-matrix}
	Let $A, B\in\mathbb{R}^{n\times n}$ be two symmetric positive definite matrices. If $A^k=B^k$ for some positive integer $k$, then $A=B$.
\end{lemma}

\begin{proof}
	Let $\lambda_i$ and $x_i$ for $1\leq i\leq n$ be the eigenvalues and corresponding eigenvectors of $A$. It follows from $A^k=B^k$ that
	\begin{equation*}
		0=(B^k-\lambda_i^k I)x_i=\sum_{t=0}^{k-1}\lambda_i^{t}B^{k-1-t}(B-\lambda_iI)x_i,\quad i=1,\ldots,n.
	\end{equation*}
	which yields $(B-\lambda_iI)x_i=0$ for any $1\leq i\leq n$ since the matrix $\sum_{t=0}^{k-1}\lambda_i^{t}B^{k-1-t}$ is symmetric positive definite. This means $(\lambda_i,x_i)$ is also an eigenpair of $B$. Since $A$ and $B$ are symmetric positive definite, their spectral decompositions have the same rank-one decomposition, thus $A=B$.
\end{proof}

We will now present some properties of symmetric positive definite matrices with block diagonal matrix forms.

\begin{lemma}\label{lemma:unique-form-1k-matrix}
	Let  $A\in\mathbb{R}^{n\times n}$ be a symmetric positive definite matrix with the eigenvalue decomposition $A=USU^T$. If $A$ has the form $(n_1,\dots,n_r)$, then $US^{1/k}U^T$ and $US^{-1/k}U^T$ with  some positive integer $k$ also have the form $(n_1,\dots,n_r)$.
\end{lemma}

\begin{proof}
	Denote the form of $A$ as ${\rm diag}(\triangle_i,1\leq i\leq r)$. Then $\triangle_i$ is also symmetric positive definite, and ${\rm diag}(\triangle_i^{1/k},1\leq i\leq r)$ has the form $(n_1,\dots,n_r)$, whose $k$th power is $A$. Since $(US^{1/k}U^T)^k=USU^T=A$, by Lemma \ref{lemma:unique-1k-matrix}, $US^{1/k}U^T={\rm diag}(\triangle_i^{1/k},1\leq i\leq r)$ has the form $(n_1,\dots,n_r)$. On the other hand, since the inverse of a matrix is unique, $US^{-1/k}U^T={\rm diag}(\triangle_i^{-1/k},1\leq i\leq r)$ also has the form $(n_1,\dots,n_r)$.
\end{proof}

\begin{lemma}\label{lemma:unique-form-log-matrix}
	Let  $A\in\mathbb{R}^{n\times n}$ be a symmetric positive definite matrix with the eigenvalue decomposition $A=USU^T$. If $U\hat{S}SU^T$ has the form $(n_1,\dots,n_r)$, in which $\hat{S}={\rm diag}(1/(S_{ii}+\epsilon S_{ii}^{1/2}))$ or $\hat{S}={\rm diag}(1/(S_{ii}+\epsilon))$, then $US^{-1}\hat{S}U^T$ and $USU^T$ also have the form $(n_1,\dots,n_r)$.
\end{lemma}

\begin{proof}
    First, let $U\hat{S}SU^T$ with  $\hat{S}={\rm diag}(1/(S_{ii}+\epsilon S_{ii}^{1/2}))$ have the form $(n_1,\dots,n_r)$. From the definition of $\hat{S}$ we know every element in ${\rm diag}(\hat{S}S)$ less than 1. Hence, $U(((\hat{S}S)^{-1}-I)/\epsilon)U^T$ is positive definite. By Lemma \ref{lemma:unique-form-1k-matrix} we know $U(\hat{S}S)^{-1}U^T$, $U(((\hat{S}S)^{-1}-I)/\epsilon)U^T$ and $U(((\hat{S}S)^{-1}-I)/\epsilon)^{-2}U^T$ all have the form $(n_1,\dots,n_r)$ by sequently. By simple calculation, it is clear that $S=(((\hat{S}S)^{-1}-I)/\epsilon)^{-2}$, which shows $A=USU^T$ has the form $(n_1,\dots,n_r)$. Similarly, By Lemma \ref{lemma:unique-form-1k-matrix} we know $US^{1/2}U^T$ has the same form as $USU^T$, $(n_1,\dots,n_r)$. Actually, by Galois theory, the roots of a quartic equation can be expressed using algebraic operations involving addition, subtraction, multiplication, division, exponentiation, and roots. This means that we could denote the solution of equation $x^4+\epsilon x^3=y$ as $$x=\hat{f}(y, \epsilon, +,-,\times, \div, \sqrt[k]{\cdot}, \text{\^{}}).$$
    Also, we know $$U\hat{S}S^{-1}U^T=U\hat{f}(S^{1/2}, \epsilon, +,-,\times, \div, \sqrt[k]{\cdot}, \text{\^{}})U^T.$$
    Therefore, by Lemma \ref{lemma:unique-form-1k-matrix},  $U\hat{S}S^{-1}U^T$ has form $(n_1,\dots,n_r)$. The proof for the case of $\hat{S}={\rm diag}(1/(S_{ii}+\epsilon))$ is similar.
\end{proof}

Let $\mathbb{Z}$ denote the set of nonnegative integers. We present the main result of this section in the following theorem.

\begin{theorem}\label{theorem51reweight}
	Let $W\in\mathbb{R}^{n\times n}$ and $\Sigma\in\mathbb{R}^{r\times r}$ be diagonal matrices with non-decreasing diagonal elements, i.e.,
	\begin{equation*}
		W_{11}\geq W_{22}\geq\cdots\geq W_{nn}>0,\quad  \Sigma_{11}\geq \Sigma_{22}\geq\cdots\geq \Sigma_{rr}>0,
	\end{equation*}
	and let $\tilde{P}\in\mathbb{R}^{n\times r}$ be a column orthogonal matrix ($n\geq r$). If $(2-q)/q\in \mathbb{Z}$, then
	\begin{equation*}
		\|W^{-1}\tilde{P}\Sigma\|_{S_q}^{q}\geq \sum_{i=1}^{r}\left(\frac{\Sigma_{ii}}{W_{ii}}\right)^{q}.
	\end{equation*}
    and
    \begin{equation*}
		\|W^{-1}\tilde{P}\Sigma\|_{L}^{1/k}\geq \sum_{i=1}^{r}\log\left(\left(\frac{\Sigma_{ii}}{W_{ii}}\right)^{1/k}+\epsilon\right),\ k=1,2.
	\end{equation*}
\end{theorem}

\begin{proof}
	Let $h(\tilde{P})=\|W^{-1}\tilde{P}\Sigma\|_{S_q}^{q}$ and ${\rm St}(n,r)=\{ \tilde{P}\in\mathbb{R}^{n\times r}|\ \tilde{P}^{T}\tilde{P}=I_{r} \}$. We define
	\begin{equation*}
		P\in\mathop{\arg\min}_{\tilde{P}\in St(n,r)}h(\tilde{P}).
	\end{equation*}
Then $0\in \partial h(P)$ from Lemma \ref{lemma: manifold3}. 

Note that the function $h(\tilde{P})=p\circ \sigma(W^{-1}\tilde{P} \Sigma)$, where $p(x)=\sum_{i}|x_i|^q$ or $p(x)=\sum_{i}\log(x_i^{1/k}+\epsilon)$. Since $W,\Sigma$ have non-zero diagonal elements, the matrix $W^{-1}P\Sigma$ has linearly independent columns, i.e., full column rank, thus $\sigma(W^{-1}P\Sigma)$ has no zero elements. This means the function $p$ is differentiable and Lipschitz continuous in the neighborhood of $\sigma(W^{-1}P\Sigma)$. By \cite[Corollary 7.6]{lewis2005nonsmooth} we know $(p\circ \sigma)$ is Fréchet derivative and Lipschitz continuous on $W^{-1}P\Sigma$. Therefore, It follows from Lemma \ref{lemma:diff-of-sigma} that
    \begin{equation*}
        \partial (p\circ \sigma)=\{U\nabla p(\sigma(W^{-1}P\Sigma))V^T\},
    \end{equation*}
    where $W^{-1}P\Sigma=USV^T$ is the SVD of $W^{-1}P\Sigma$. Thus, we have
	\begin{equation}\label{lemma1need5}
		\partial h(P) = \{W^{-1}U\hat{p}\hat{S}V^T\Sigma\},
	\end{equation}
    where 
    \begin{equation}\label{eq:lemma-new1}
        \hat{p}=\left\{\begin{aligned}
            &q, &&\text{ if }p(x)=\sum_{i}|x_i|^q,\\
            &\frac{1}{k}, &&\text{ if }p(x)=\sum_{i}\log(x_i^{1/k}+\epsilon),
        \end{aligned}\right.,
        \ \hat{S}=\left\{\begin{aligned}
            &S^{q-1}, &&\text{ if }p(x)=\sum_{i}|x_i|^q,\\
            &{\rm diag}\left(\frac{1}{S_{ii}+\epsilon S_{ii}^{1-1/k}}\right), &&\text{ if }p(x)=\sum_{i}\log(x_i^{1/k}+\epsilon).
        \end{aligned}\right.
    \end{equation}
	We denote $F=\hat{p}U\hat{S}V^T$. Also, by Lemma \ref{lemma: manifold2}, $h$ is regular along $T_P{\rm St}(n,r)$, and by Lemma \ref{lemma: manifold1}, we have
	\begin{equation}\label{lemma1need1}
		0\in \partial h(P)={\rm Proj}_{T_P{\rm St}(n,r)}(\partial h(P)),
	\end{equation}
	using the representation of the tangent space on the Stiefel manifold from \eqref{eq:manifold-st-tanget}. From \eqref{lemma1need1}, we have
	\begin{equation}\label{lemma1need2}
		0\in\mathop{\arg\min}_{Z\in\mathbb{R}^{n\times r}}\|Z-W^{-1}F\Sigma\|_F^2,\ \text{s.t.}\  P^TZ+Z^TP=0.
	\end{equation}
	The KKT condition of \eqref{lemma1need2} is $\exists \Lambda\in\mathbb{R}^{r\times r}$ such that
	\begin{equation*}
		2(Z^{T}-W^{-1}F\Sigma)+P\Lambda+P\Lambda^T=0,\quad P^TZ+Z^TP=0.
	\end{equation*}
	By \eqref{lemma1need2}, we know $Z^{T}=0$ is an optimal solution, thus we can deduce
	\begin{equation}\label{lemma1need6}
		W^{-1}F\Sigma=P\frac{\Lambda+\Lambda^T}{2}.
	\end{equation}
	From \eqref{lemma1need5} and \eqref{eq:lemma-new1}, we have
	\begin{equation}\label{lemma1need7}
		F=\hat{p}U\hat{S}V^T=\hat{p}USV^TV\hat{S}S^{-1}V^T=\hat{p}W^{-1}P\Sigma V\hat{S}S^{-1}V^T,
	\end{equation}
	and
	\begin{equation}\label{lemma1need8}
		F=\hat{p}U\hat{S}V^T=\hat{p}US^{-1}\hat{S}U^TUSV=\hat{p}US^{-1}\hat{S}U^TW^{-1}P\Sigma.
	\end{equation}
	Combining \eqref{lemma1need6} and \eqref{lemma1need8}, it holds
	\begin{equation}\label{lemma1need9}
		\hat{p}P^TW^{-1}US^{-1}\hat{S}U^TW^{-1}P\Sigma^{2}=\frac{\Lambda+\Lambda^T}{2},
	\end{equation}
 which is a symmetric matrix.
	Note that $\Sigma$ is classified by singular values into the form $(n_1,\dots,n_s)$, i.e.,
	\begin{equation*}
		\Sigma=\begin{bmatrix}
			\sigma_1 I_1 &  &  &  \\
			& \sigma_2 I_2 &  &  \\
			&  & \ddots &  \\
			&  &  & \sigma_s I_s 
		\end{bmatrix},
	\end{equation*}
	where $\sigma_1>\sigma_2>\cdots>\sigma_s$, and $I_1,\dots, I_s$ are $n_1,\dots,n_s$-dimensional identity matrices, respectively. It follows from \eqref{lemma1need9} and Lemma \ref{lemma:sym-and-diag}  that the matrix $P^TW^{-1}US^{-1}\hat{S}U^TW^{-1}P$ has the same block-diagonal form, $(n_1,\dots,n_s)$, as $\Sigma$, since $S^{-1}\hat{S}$ is diagonal. Therefore, we can exchange $\Sigma$ with $P^TW^{-1}US^{-1}\hat{S}U^TW^{-1}P$, i.e.,
	\begin{equation}\label{lemma1need11}
		\frac{\Lambda+\Lambda^T}{2}=\hat{p}\Sigma P^TW^{-1}US^{-1}\hat{S}U^TW^{-1}P\Sigma.
	\end{equation}
	Also, it holds
	\begin{equation}\label{lemma1need10}
		U^TW^{-1}P=U^TW^{-1}P\Sigma \Sigma^{-1}=U^T(USV^T)\Sigma^{-1}=SV^T\Sigma^{-1}.
	\end{equation}
	Combining \eqref{lemma1need11} and \eqref{lemma1need10}, we can see that
    \begin{equation}\label{lemma1need14}
		\frac{\Lambda+\Lambda^T}{2}=\hat{p}\Sigma P^TW^{-1}US^{-1}\hat{S}U^TW^{-1}P\Sigma=\hat{p}VS\hat{S}V^T,
	\end{equation}
 which is invertible.
	When $p(x)=\sum_{i}\log(x_i^{1/2}+\epsilon)$ or $p(x)=\sum_{i}\log(x_i+\epsilon)$, by Lemma \ref{lemma:unique-form-log-matrix}, we know that $VS^{-1}\hat{S}V^T$ is invertible and has the same block-diagonal form $(n_1,\dots,n_s)$ as $P^TW^{-1}US^{-q}U^TW^{-1}P$ and $\Sigma$. When $p(x)=\sum_{i}|x_i|^q$, $S\hat{S}=S^q$, by Lemma \ref{lemma:unique-form-1k-matrix} and $(2-q)/q\in\mathbb{Z}$, we know that $VS^{-1}\hat{S}V^T=VS^{q-2}V^T,\Sigma$ have the same form $(n_1,\dots,n_s)$. Therefore,
	\begin{equation}\label{lemma1need12}
		VS^{-1}\hat{S}V^T\Sigma=\Sigma VS^{-1}\hat{S}V^T.
	\end{equation}
	Similarly, from \eqref{lemma1need6}, \eqref{lemma1need7}, and \eqref{lemma1need12}, we have
	\begin{equation*}
		P\frac{\Lambda+\Lambda^T}{2}P^T=\hat{p}W^{-2}P\Sigma VS^{-1}\hat{S}V^T\Sigma P^T=\hat{p}W^{-2}P\Sigma^{2} VS^{-1}\hat{S}V^T P^T.
	\end{equation*}
	It follows from \eqref{lemma1need14} that
	\begin{equation*}
		PVS\hat{S}V^TP^T=W^{-2}P\Sigma^{2}VS^{-1}\hat{S}V^T P^T.
	\end{equation*}
	Noting that $P^TP=I$, moving the right-hand side to the left and multiplying by $P^T$, we can see that
	\begin{equation*}
		W^{-2}P\Sigma^{2} P^T=PVS^{2}V^TP^T
	\end{equation*}
  is symmetric.
	By Lemma \ref{lemma:sym-and-diag}, we know $P\Sigma^{2}P^T$ has the same block-diagonal form as $W^{-1}$, thus $W^{-1}P\Sigma^{2}P^T$ is symmetric. Define
	\begin{equation*}
		\hat{P}=[P,P^{\bot}],
	\end{equation*}
	where $P^{\bot}$ is the complementary orthogonal basis of $P$. Therefore,
	\begin{equation*}
		\hat{P}^TW^{-1}\hat{P}\begin{bmatrix}\Sigma^2 & \\ & 0\end{bmatrix}=\hat{P}^TW^{-1}\hat{P}\begin{bmatrix}\Sigma^2 & \\ & 0\end{bmatrix}\hat{P}^T\hat{P}=\hat{P}^TW^{-1}P\Sigma^2P^T\hat{P}
	\end{equation*} is symmetric. 
	By Lemma \ref{lemma:sym-and-diag}, we know $\hat{P}^TW^{-1}\hat{P}$ has the same block-diagonal form as ${\rm diag}(\Sigma^2,0)$, denoted as $(n_1,\dots,n_s,n_{s+1})$, where $n_{s+1}$ is the size of the zero matrix. Therefore,
	\begin{equation*}
		\hat{P}^TW^{-1}\hat{P}\begin{bmatrix}\Sigma & \\ & 0\end{bmatrix}
	\end{equation*}
	is symmetric and
	\begin{equation*}
		\hat{P}^TW^{-1}\hat{P}=\begin{bmatrix}
			\tilde{\triangle}_1 &  &  &  \\
			& \tilde{\triangle}_2 &  &  \\
			&  & \ddots &  \\
			&  &  & \tilde{\triangle}_{s+1} 
		\end{bmatrix},
	\end{equation*}
	where the size of $\tilde{\triangle}_i$ is the same $n_i$ ($1\leq i\leq s+1$). Since the Schatten-$q$ norm is a unitarily invariant function, we have
	\begin{align}\label{lemma1need13}
		\| W^{-1}P\Sigma\|_{S_q}^{q}&=\|\hat{P}^TW^{-1}\hat{P}\begin{bmatrix}\Sigma & \\ & 0\end{bmatrix}\|_{S_q}^{q}=\left\|\begin{bmatrix}
			\sigma_1 \tilde{\triangle}_1 &  &  &  \\
			&  \ddots & &  \\
			&  &  \sigma_s \tilde{\triangle}_s & \\
			& & & 0
		\end{bmatrix}\right\|_{S_q}^{q}\nonumber\\
	&=\sum_{i=1}^s\sigma_i\|\tilde{\triangle}_i\|_{S_q}^{q}.
	\end{align}
	Since a unitary transformation does not change the singular values of a positive definite matrix, the singular values of $\tilde{\triangle}_i$ ($1\leq i\leq s$) are $r$ elements of $(W_{11}^{-1},W_{22}^{-1},\dots,W_{nn}^{-1})$. If we set $\Sigma_{ii}=0,\ r<i\leq n$, then \eqref{lemma1need13} can be rewritten as
	\begin{equation*}
		\| W^{-1}P\Sigma\|_{S_q}^{q}=\sum_{i=1}^{n}\left(\frac{\Sigma_{ii}}{W_{\pi(i)\pi(i)}}\right)^q\geq \sum_{i=1}^{r}\left(\frac{\Sigma_{ii}}{W_{ii}}\right)^q,
	\end{equation*}
	where $\pi$ is a permutation operator, and the last inequality comes from the rearrangement inequality.\\
    Similarly, since Logarithmic Norm is a unitarily invariant function, we have 
    \begin{align}\label{lemma1need31}
		\| W^{-1}P\Sigma\|_{L}^{1/2}&=\|\hat{P}^TW^{-1}\hat{P}\begin{bmatrix}\Sigma & \\ & 0\end{bmatrix}\|_{L}^{1/2}=\sum_{i=1}^s\|\sigma_i\tilde{\triangle}_i\|_{L}^{1/2}.
	\end{align}
    Let $f_i(x)=\log(\left(\Sigma_{ii}/x\right)^{1/2}+\epsilon)$ for $x>0$, we know $f'_{i+1}(x)-f'_{i}(x)\geq0$ by simple calculation. Then from the generalization of rearrangement inequality in \cite{holstermann2017generalization} and \eqref{lemma1need31}, we can get
    \begin{equation*}
		\| W^{-1}P\Sigma\|_{L}^{1/2}=\sum_{i=1}^{r}\log\left(\left(\frac{\Sigma_{ii}}{W_{\pi(i)\pi(i)}}\right)^{1/2}+\epsilon\right)\geq \sum_{i=1}^{r}\log\left(\left(\frac{\Sigma_{ii}}{W_{ii}}\right)^{1/2}+\epsilon\right),
	\end{equation*}
    where $\pi$ is a permutation operator.
\end{proof}

\section{Factorization Forms of Weighted Schatten-q Norms}\label{sec:re-3}\label{sec:re-4}

In this section, we establish the specific factorization forms of the weighted Schatten-$q$ norms with $q=1,0.5,2/3$, and the weighted Logarithmic norm using the factorization strategy provided in the previous section.

To simplify notation and provide clear definitions, we present the expressions for the weighted Schatten-$q$ norms and the weighted Logarithmic norm.

\begin{definition}
	Let $X\in\mathbb{R}^{n\times m}$,  $q>0$, and $W={\rm diag}(W_{11},W_{22},\dots,W_{tt})$ with $t=\min\{m,n\}$ be a given descending weighted diagonal matrix. Let $X=U\Sigma V^{T}$ be the SVD of $X$. Then, the corresponding weighted Schatten-$q$ norm of $X$ is defined as
	\begin{equation*}
		\|X\|_{S_q,W}=\left(\sum_{i=1}^{t}\frac{\Sigma_{ii}^q}{W_{ii}^q}\right)^{1/q}.
	\end{equation*}
The corresponding weighted Logarithmic norm is defined as
 \begin{equation*}
		\|X\|_{L,W}^{1/2}=\sum_{i=1}^{t}\log\left(\left(\frac{\Sigma_{ii}}{W_{ii}}\right)^{1/2}+\epsilon\right).
	\end{equation*}
\end{definition}
Here, if $W=I$ and the domain is the real field, the above definitions are the Schatten-q norm and the  Logarithmic norm. 
 
We will  prove that  the factorization forms of the weighted Schatten-q norms for $q=1,0.5,2/3$ and Logarithmic norm have explicite expressions.

\begin{theorem}\label{thm51}
	 Let $W\in\mathbb{R}^{r\times r}$ be a non-decreasing diagonal matrix, i.e.,
	\begin{equation*}
		W_{11}\geq W_{22}\geq\cdots\geq W_{rr}>0.
	\end{equation*}
	For $X\in\mathbb{R}^{m\times n}$ with singular values $\sigma_i\ (1\leq i\leq {\rm rank}(X))$, if $r\geq {\rm rank}(X)$, then we have
	\begin{itemize}
		\item [\rm (1)] \begin{equation*}
			\sum_{i=1}^{{\rm rank}(X)}\frac{\sigma_i}{W_{ii}}=\min_{A\in\mathbb{R}^{m\times r}, B\in \mathbb{R}^{n\times r}, X=AB^{T}}\frac{1}{2}(\|AW^{-1}\|_F^2+\|B\|_F^2),
		\end{equation*}
		\item [\rm (2)] \begin{equation*}
			\sum_{i=1}^{{\rm rank}(X)}\left(\frac{\sigma_i}{W_{ii}}\right)^{\frac{1}{2}}=\min_{A\in\mathbb{R}^{m\times r}, B\in \mathbb{R}^{n\times r}, X=AB^{T}}\frac{1}{2}(\|AW^{-1}\|_{*}+\|B\|_{*}),
		\end{equation*}
		\item [\rm (3)] \begin{equation*}
			\sum_{i=1}^{{\rm rank}(X)}\left(\frac{\sigma_i}{W_{ii}}\right)^{\frac{2}{3}}=\min_{A\in\mathbb{R}^{m\times r}, B\in \mathbb{R}^{n\times r}, X=AB^{T}}\frac{1}{3}(\|AW^{-1}\|_F^2+2\|B\|_{*}),
		\end{equation*}
        \item [\rm (4)] \begin{equation*}
			\|X\|_{L,W}^{1/2}=\min_{A\in\mathbb{R}^{m\times r}, B\in \mathbb{R}^{n\times r}, X=AB^{T}}\frac{1}{2}(\|AW^{-1}\|_{L}^1+\|B\|_{L}^1).
		\end{equation*}
	\end{itemize}
	\end{theorem}
\begin{proof}
	The singular value decomposition of the matrix $X$ is given by
	\begin{equation*}
		X=U\Sigma_{X}V^{T}.
	\end{equation*}
	First, we prove that the left-hand side of conclusions (1), (2), and (3) respectively are greater than or equal to the right-hand side. Let $W_X$ be the matrix obtained by extracting the first ${\rm rank}(X)$ rows and columns of $W$. Define
	\begin{equation*}
		\tilde{A}=\begin{bmatrix}U\Sigma_{X}^{\frac{1}{2}}W_{X}^{\frac{1}{2}}&0\end{bmatrix}
		,\ \tilde{B}=\begin{bmatrix}V\Sigma_{X}^{\frac{1}{2}}W_{X}^{-\frac{1}{2}}&0\end{bmatrix},
	\end{equation*}
	then for cases (1), (2) and (4), we have
	\begin{itemize}
		\item [(1)] \begin{equation*}
			\sum_{i=1}^{{\rm rank}(X)}\frac{\sigma_i}{W_{ii}}=\frac{1}{2}(\|\tilde{A}W^{-1}\|_F^2+\|\tilde{B}\|_F^2)\geq\min_{X=AB^{T}}\frac{1}{2}(\|AW^{-1}\|_F^2+\|B\|_F^2).
		\end{equation*}
		\item [(2)] \begin{equation*}
			\sum_{i=1}^{{\rm rank}(X)}\left(\frac{\sigma_i}{W_{ii}}\right)^{\frac{1}{2}}=\frac{1}{2}(\|\tilde{A}W^{-1}\|_{*}+\|\tilde{B}\|_{*})\geq\min_{X=AB^{T}}\frac{1}{2}(\|AW^{-1}\|_{*}+\|B\|_{*}).
		\end{equation*}
        \item [(4)] \begin{equation*}
			\|X\|_{L,W}^{1/2}=\frac{1}{2}(\|\tilde{A}W^{-1}\|_{L}^1+\|\tilde{B}\|_{L}^1)\geq\min_{X=AB^{T}}\frac{1}{2}(\|AW^{-1}\|_{L}^1+\|B\|_{L}^1).
		\end{equation*}
	\end{itemize}
	For case (3), we define
	\begin{equation*}
		\hat{A}=\begin{bmatrix}U\Sigma_{X}^{\frac{1}{3}}W_{X}^{\frac{2}{3}}&0\end{bmatrix}
		,\ \tilde{B}=\begin{bmatrix}V\Sigma_{X}^{\frac{2}{3}}W_{X}^{-\frac{2}{3}}&0\end{bmatrix},
	\end{equation*}
	then we have
	\begin{equation*}
		\sum_{i=1}^{{\rm rank}(X)}\left(\frac{\sigma_i}{W_{ii}}\right)^{\frac{2}{3}}=\frac{1}{3}(\|\hat{A}W^{-1}\|_{F}^2+2\|\hat{B}\|_{*})\geq\min_{X=AB^{T}}\frac{1}{3}(\|AW^{-1}\|_{F}^2+2\|B\|_{*}).
	\end{equation*}
	Now, for any $A,B$ such that $AB^{T}=X$, let their singular value decompositions be:
	\begin{equation*}
		A=U_{A}\Sigma_{A}V_{A}^{T},\ B=U_{B}\Sigma_{B}V_{B}^{T}.
	\end{equation*}
	Here, $V_A$ is a square matrix. It can be derived that
	\begin{align}\label{eqn:thm1-need1}
		U_A\Sigma_AV_A^{T}V_B\Sigma_BU_B^{T}=X \Longrightarrow\Sigma_{A}V_{A}^{T}V_{B}\Sigma_{B}=U_{A}^{T}U\Sigma_{X}V^{T}U_{B}.
	\end{align}
	Also,
	\begin{equation*}
		U_{A}^{T}U=U_{A}^{T}U\Sigma_{X}V^{T}V\Sigma_{X}^{-1}=\Sigma_{A}V_{A}^{T}B^{T}V\Sigma_{X}^{-1},
	\end{equation*}
	which implies
	\begin{align*}
		U^{T}U_AU_{A}^{T}U=&\Sigma_X^{-1}V^{T}BV_A\Sigma_A\Sigma_AV_A^{T}B^{T}V\Sigma_{X}^{-1}\\
  =&\Sigma_X^{-1}V^{T}BV_A\Sigma_AU_A^{T}U_A\Sigma_AV_A^{T}B^{T}V\Sigma_{X}^{-1}\\
  =&\Sigma_X^{-1}V^{T}X^{T}XV\Sigma_{X}^{-1}=I
	\end{align*}
 `This means $U_{A}^{T}U$ is columns orthogonal matrix, and also we know $U_B^{T}V$ is columns orthogonal matrix too.\\
	Now, let's prove conclusions (1), (2), and (3):\\
	{\bf Conclusion (1):}\\
	Since the Frobenius norm is a unitarily invariant function, we have
	\begin{align*}
		\frac{1}{2}(\|AW^{-1}\|_F^2+\|B\|_F^2)=&\frac{1}{2}(\|\Sigma_{A}V_{A}^{T}W^{-1}\|_F^2+\|\Sigma_{B}\|_F^2)\\
		\geq &\frac{1}{2}(\|\Sigma_{A}W_{A}^{-1}\|_F^2+\|\Sigma_{B}\|_F^2)\\
		=&\frac{1}{2}(\|W_{A}^{-1}\Sigma_{A}\|_F^2+\|V_{A}^{T}V_{B}\Sigma_{B}\|_F^2)\\
		\geq &\|W_{A}^{-1}\Sigma_{A}V_{A}^{T}V_{B}\Sigma_{B}\|_{*},
	\end{align*}
	where the first inequality comes from case $q=2$ in Theorem \ref{theorem51reweight} and $W_A={\rm diag}(W_{ii},1\leq i\leq{\rm rank}(\Sigma_A))$, the second equality is due to \( V_A^{T} \) being an orthogonal matrix, and the last inequality comes from \eqref{eq:BM}. According to \eqref{eqn:thm1-need1}, we know that
	\begin{equation}\label{eqn:thm1-need2}
		\frac{1}{2}(\|AW^{-1}\|_F^2+\|B\|_F^2)\geq \|W_{A}^{-1}U_{A}^{T}U\Sigma_{X}\|_{*}.
	\end{equation}
	Since \( U_{A}^{T}U \) is a column orthogonal matrix, by case $q=1$ in Theorem \ref{theorem51reweight} and \eqref{eqn:thm1-need2}, we have
	\begin{equation*}
		\frac{1}{2}(\|AW^{-1}\|_F^2+\|B\|_F^2)\geq \|W_{A}^{-1}U_{A}^{T}U\Sigma_{X}\|_{*}\geq \sum_{i=1}^{{\rm rank}(X)}\frac{\sigma_i}{W_{ii}}.
	\end{equation*}
	Combining the fact that the left-hand side of Conclusion (1) is greater than or equal to the right-hand side, this proves Conclusion (1).\\
	{\bf Conclusion (2):}\\
	Since the nuclear norm is a unitarily invariant function, we have
	\begin{align*}
		\frac{1}{2}(\|AW^{-1}\|_*+\|B\|_*)=&\frac{1}{2}(\|\Sigma_{A}V_{A}^{T}W^{-1}\|_*+\|\Sigma_{B}\|_*)\\
		\geq &\frac{1}{2}(\|\Sigma_{A}W_{A}^{-1}\|_*+\|\Sigma_{B}\|_*)\\
		=&\frac{1}{2}(\|W_{A}^{-1}\Sigma_{A}\|_*+\|V_{A}^{T}V_{B}\Sigma_{B}\|_*)\\
		\geq &\|W_{A}^{-1}\Sigma_{A}V_{A}^{T}V_{B}\Sigma_{B}\|_{S_{\frac{1}{2}}}^{\frac{1}{2}},
	\end{align*}
	where the first inequality comes from by case $q=1$ in Theorem \ref{theorem51reweight}, the second equality is due to \( V_A^{T} \) being an orthogonal matrix, and the last inequality comes from factorization form of Schatten-0.5 norm in Lemma \ref{lemma:schatten-12-23}. According to \eqref{eqn:thm1-need1} and \( U_{B}^{T}V \) is a column orthogonal matrix, we know that
	\begin{equation}\label{eqn:thm1-need2-2}
		\frac{1}{2}(\|AW^{-1}\|_*+\|B\|_*)\geq \|W_{A}^{-1}U_{A}^{T}U\Sigma_{X}\|_{S_{\frac{1}{2}}}^{\frac{1}{2}}\geq \sum_{i=1}^{{\rm rank}(X)}\left(\frac{\sigma_i}{W_{ii}}\right)^{\frac{1}{2}},
	\end{equation}
	where the last inequality holds since \( U_{A}^{T}U \) is a column orthogonal matrix and case $q=0.5$ in Theorem \ref{theorem51reweight} holds. Combining the fact that the left-hand side of Conclusion (2) is greater than or equal to the right-hand side, this proves Conclusion (2).\\
	{\bf Conclusion (3):}\\
	Since the nuclear norm and Frobenius norm are unitarily invariant functions, we have
	\begin{align*}
		\frac{1}{3}(\|AW^{-1}\|_F^2+2\|B\|_*)=&\frac{1}{3}(\|\Sigma_{A}V_{A}^{T}W^{-1}\|_F^2+2\|\Sigma_{B}\|_*)\\
		\geq &\frac{1}{3}(\|\Sigma_{A}W_{A}^{-1}\|_F^2+2\|\Sigma_{B}\|_*)\\
		=&\frac{1}{3}(\|W_{A}^{-1}\Sigma_{A}\|_F^2+2\|V_{A}^{T}V_{B}\Sigma_{B}\|_*)\\
		\geq &\|W_{A}^{-1}\Sigma_{A}V_{A}^{T}V_{B}\Sigma_{B}\|_{S_{\frac{2}{3}}}^{\frac{2}{3}},
	\end{align*}
	where the first inequality comes from case $q=2$ in Theorem \ref{theorem51reweight}, the second equality is due to \( V_A^{T} \) being an orthogonal matrix, and the last inequality comes from factorization form of Schatten-2/3 norm in Lemma \ref{lemma:schatten-12-23}. According to \eqref{eqn:thm1-need1} and \( U_{B}^{T}V \) is a column orthogonal matrix, we know that
	\begin{equation}\label{eqn:thm1-need2-3}
		\frac{1}{3}(\|AW^{-1}\|_F^2+2\|B\|_*)\geq \|W_{A}^{-1}U_{A}^{T}U\Sigma_{X}\|_{S_{\frac{2}{3}}}^{\frac{2}{3}}\geq \sum_{i=1}^{{\rm rank}(X)}\left(\frac{\sigma_i}{W_{ii}}\right)^{\frac{2}{3}},
	\end{equation}
	where the last inequality holds since \( U_{A}^{T}U \) is a column orthogonal matrix and case $q=2$ in Theorem \ref{theorem51reweight} holds. Combining the fact that the left-hand side of Conclusion (3) is greater than or equal to the right-hand side, this proves Conclusion (3).\\
    {\bf Conclusion (4):}\\
    Since the Logarithmic norm and Frobenius norm are unitarily invariant functions, we have
	\begin{align*}
		\frac{1}{2}(\|AW^{-1}\|_{L}^1+\|B\|_{L}^1)=&\frac{1}{2}(\|\Sigma_{A}V_{A}^{T}W^{-1}\|_{L}^1+\|\Sigma_{B}\|_{L}^1)\\
		\geq &\frac{1}{2}(\|\Sigma_{A}W_{A}^{-1}\|_{L}^1+\|\Sigma_{B}\|_{L}^1)\\
		=&\frac{1}{2}(\|W_{A}^{-1}\Sigma_{A}\|_{L}^1+\|V_{A}^{T}V_{B}\Sigma_{B}\|_{L}^1)\\
		\geq &\|W_{A}^{-1}\Sigma_{A}V_{A}^{T}V_{B}\Sigma_{B}\|_{L}^{1/2},
	\end{align*}
    where the first inequality comes from Logarithmic norm case in Theorem \ref{theorem51reweight} with $k=1$, the second equality is due to \( V_A^{T} \) being an orthogonal matrix, and the last inequality comes from Lemma \ref{lemma:logdet-BM}. According to \eqref{eqn:thm1-need1} and \( U_{B}^{T}V \) is a column orthogonal matrix, we know that
	\begin{equation*}
		\frac{1}{2}(\|AW^{-1}\|_{L}^1+\|B\|_{L}^1)\geq \|W_{A}^{-1}U_{A}^{T}U\Sigma_{X}\|_{L}^{1/2}\geq\|X\|_{L}^{1/2},
	\end{equation*}
    where the last inequality is from Logarithmic norm case in Theorem \ref{theorem51reweight} with $k=2$.
\end{proof}

\section{Algorithms and convergence result}
We replace the nuclear norm in \eqref{eq:RPCA-model} with the factorization forms of weighted Schatten-$q$ norms or the Logarithmic norm, where the non-convex representation enhances the low-rank recovery capability of the model. The model is stated as follows:
\begin{equation}\label{eq:rpca-decomposition}
    \min\lambda(h_1(AW^{-1})+h_2(B))+\|S\|_1,\ {\rm s.t.}\quad AB^T=X,\ P_{\Omega}(X+S-M)=0.
\end{equation}
Here, $h_1$ and $h_2$ have different definitions under different non-convex representations.
\begin{align*}
    &\text{weighted nuclear norm: }h_1(AW^{-1})=\frac{1}{2}\|AW^{-1}\|_F^2,\ h_2(B)=\frac{1}{2}\|B\|_F^2,\nonumber\\
    &\text{weighted Schatten-0.5 norm: }h_1(AW^{-1})=\frac{1}{2}\|AW^{-1}\|_{*},\ h_2(B)=\frac{1}{2}\|B\|_{*},\nonumber\\
    &\text{weighted Schatten-2/3 norm: }h_1(AW^{-1})=\frac{1}{3}\|AW^{-1}\|_F^2,\ h_2(B)=\frac{2}{3}\|B\|_{*},\nonumber\\
    &\text{weighted Logarithmic norm: }h_1(AW^{-1})=\frac{1}{2}\|AW^{-1}\|_L^1,\ h_2(B)=\frac{1}{2}\|B\|_L^1.
\end{align*}
We use an ADMM framework algorithm to solve \eqref{eq:rpca-decomposition}. In order to ensure explicit solutions for subproblems within this framework, we employ two relaxation variables, and then \eqref{eq:rpca-decomposition} could be rewritten as
\begin{equation}\label{eq:rpca-decomposition-2}
    \min\lambda(h_1(\hat{A})+h_2(\hat{B}))+\|P_{\Omega}(S)\|_1,\ {\rm s.t.}\quad AB^T=X,\ X+S-M=0,\ \hat{A}=AW^{-1},\ \hat{B}=B.
\end{equation}
The corresponding augmented Lagrangian function is:
\begin{align}\label{eq: model-l1-ALS-new}
    L_{\rho}(A,B,S,X,\hat{A},\hat{B};Y_1,Y_2,Y_3,Y_4)=&\lambda(h_1(\hat{A})+h_2(\hat{B}))+\|P_{\Omega}(S)\|_1+\langle Y_1,\hat{A}-AW^{-1}\rangle+\langle Y_2,\hat{B}-B\rangle\nonumber\\
    &+\langle Y_3,AB^T-X\rangle+\langle Y_4,X+S-M\rangle+\frac{\rho}{2}\|\hat{A}-AW^{-1}\|_F^2\nonumber\\
    &+\frac{\rho}{2}\left(\|\hat{B}-B\|_F^2+\|AB^T-X\|_F^2+\|X+S-M\|_F^2\right).
\end{align}
Hence we could establish the following algorithm to solve \eqref{eq:rpca-decomposition-2}:
\begin{algorithm}
    \caption{Matrix recovery based on Weighted factorization form}\label{alg:re-q-12}
    \begin{algorithmic}[1]
        \STATE \textbf{Input:} Parameters $\lambda$, observed matrix and index set $\mathcal{M},\Omega$, initial points $X^0,\hat{A}^0,\hat{B}^0,S^0, A^0, B^0,Y_i^{0},\rho^0,1\leq i\leq 3$.
        \STATE For $k=0,1,\dots$, compute the following iterative steps:
        \STATE $A^{k+1}=\mathop{\arg\min}L_{\rho^k}(X^k,S^k, A, B^k,\hat{A}^k,\hat{B}^k;Y_i^{k},1\leq i\leq 3)$.
        \STATE $B^{k+1}=\mathop{\arg\min}L_{\rho^k}(X^k,S^k, A^{k+1}, B,\hat{A}^k,\hat{B}^k;Y_i^{k},1\leq i\leq 3)$.
        \STATE $S^{k+1}=\mathop{\arg\min}L_{\rho^k}(X^k,S, A^{k+1}, B^{k+1},\hat{A}^k,\hat{B}^k;Y_i^{k},1\leq i\leq 3)$.
        \STATE $X^{k+1}=\mathop{\arg\min}L_{\rho^k}(X,S^{k+1}, A^{k+1}, B^{k+1},\hat{A}^k,\hat{B}^k;Y_i^{k},1\leq i\leq 3)$.
        \STATE $\hat{A}^{k+1}=\mathop{\arg\min}L_{\rho^k}(X^{k+1},S^{k+1}, A^{k+1}, B^{k+1},\hat{A},\hat{B}^k;Y_i^{k},1\leq i\leq 3)$.
        \STATE $\hat{B}^{k+1}=\mathop{\arg\min}L_{\rho^k}(X^{k+1},S^{k+1}, A^{k+1}, B^{k+1},\hat{A}^{k+1},\hat{B};Y_i^{k},1\leq i\leq 3)$.
        \STATE $Y^{k+1}_1=Y^{k}_1+\rho^k(\hat{A}^{k+1}-A^{k+1}W^{-1})$.
        \STATE $Y^{k+1}_2=Y^{k}_2+\rho^k(\hat{B}^{k+1}-B^{k+1})$.
        \STATE $Y^{k+1}_3=Y^{k}_3+\rho^k(A^{k+1}B^{k+1}-X^{k+1})$.
        \STATE $Y^{k+1}_4=Y^{k}_4+\rho^k(X^{k+1}+S^{k+1}-M)$.
        \STATE $\rho^{k+1}=\mu\rho^{k}$.
    \end{algorithmic}
\end{algorithm}
Actually, the updating of every variable in algorithm \ref{alg:re-q-12} is proximal point operator ${\rm Prox}_{th_1}(\cdot)$, which is 
$${\rm Prox}_{th}(X)=\mathop{\arg\min}_{Y\in\mathbb{R}^{n\times m}}\left\{ th(Y)+\frac{1}{2}\|Y-X\|_F^2 \right\}, \ \forall X\in\mathbb{R}^{n\times m}.$$
Specifically, we present the update of each variable as follows:
\begin{align*}
    &A^{k+1}=(\rho^kX^{k}(B^{k})^{T}+\rho\hat{A}^{k}W^{-1}+Y_1^{k}W^{-1}-Y_3^{k}(B^k)^{T})\times(\rho^kB^k(B^k)^{T}+\rho^kW^{-2})^{-1},\nonumber\\
    &B^{k+1}=(\rho^k(A^{k+1})^{T}A^{k+1}+\rho^kI)^{-1}\times(\rho^k(A^{k+1})^{T}X^{k}+\rho\hat{B}^{k}+Y_2^{k}-(A^{k+1})^{T}Y_3^{k}),\nonumber\\
    &S^{k+1}=P_{\Omega}({\rm Prox}_{(\rho^k)^{-1}\|\cdot\|_1}(M-X^{k}-\frac{Y_4^k}{\rho^k}))+P_{\Omega^c}(M-X^{k}-\frac{Y_4^k}{\rho^k}),\nonumber\\
    &X^{k+1}=\frac{1}{2\rho^k}(Y_3^k-Y_4^k+\rho^kA^{k+1}(B^{k+1})^{T}+S^{k+1}-M),\nonumber\\
    &\hat{A}^{k+1}={\rm Prox}_{(\rho^k)^{-1}\lambda h_1}(A^{k+1}W^{-1}-\frac{Y_1^k}{\rho^k}),\nonumber\\
    &\hat{B}^{k+1}={\rm Prox}_{(\rho^k)^{-1}\lambda h_2}(B^{k+1}-\frac{Y_2^k}{\rho^k}).
\end{align*}
Here, the analytical expression of the proximal point operator mentioned above is provided as follows:
\begin{align*}
    &{\rm Prox}_{t\|\cdot\|_1}(Y)=\max(Y-t,0),\nonumber\\
    &{\rm Prox}_{t\|\cdot\|_F^2}(Y)=\frac{Y}{1+t},\nonumber\\
    &{\rm Prox}_{t\|\cdot\|_{*}}(Y)=U{\rm diag}(\max(\sigma(Y)-t,0))V^T,\ Y=U{\rm diag}(\sigma(Y))V^T\text{ is SVD},\nonumber\\
    &{\rm Prox}_{t\|\cdot\|_L^1}(Y)=U\mathbb{L}_{t,\epsilon}({\rm diag}(\sigma(Y)))V^T,\ Y=U{\rm diag}(\sigma(Y))V^T\text{ is SVD},\nonumber\\
\end{align*}
where $\mathbb{L}_{t,\epsilon}(y)$ denotes the logarithmic singular value thresholding (LSVT) operator, which is defined in \cite{kang2016top}, and its entry-wise form is
\begin{equation*}
    \triangle=(y-\epsilon)^2-4(t-y\epsilon),\ \mathbb{L}_{t,\epsilon}(y)=\left\{\begin{aligned}
        &0, &&\triangle\leq 0,\\
        &\mathop{\arg\min}_{x\in\{0,\frac{1}{2}(y-\epsilon+\sqrt{\triangle})\}}\frac{1}{2}(x-y)^2+t\log(x+\epsilon), &&\triangle>0.
    \end{aligned}\right.
\end{equation*}
Before proving the convergence of Algorithm \ref{alg:re-q-12}, we need to demonstrate that the iterative sequences of the variables $\mathcal{A}$, $\mathcal{B}$, $\mathcal{S}$, $\mathcal{X}$, $\hat{\mathcal{A}}$, $\hat{\mathcal{B}}$, $\mathcal{Y}_i$, $1\leq i\leq 4$ are bounded.
\begin{lemma}
    Let $(A^k,B^k,S^k,X^k,\hat{A}^k,\hat{B}^k,Y_i^k,1\leq i\leq 3)$ denote the iterative sequences generated by Algorithm \ref{alg:re-q-12}. If $\{Y^k_3\}$ is bounded, then $\{(A^k,B^k,S^k,X^k,\hat{A}^k,\hat{B}^k,Y_i^k,1\leq i\leq 3)\}$ is bounded.
\end{lemma}
\begin{proof}
    The proof of the cases ``weighted nuclear norm", ``weighted Schatten-0.5 norm", and ``weighted Schatten-2/3 norm" are similar to the \cite{shang2017bilinear}, with the only difference being the inclusion of the weight matrix $W$. Therefore, we omit the corresponding proof. In the case ``weighted Logarithmic norm", from the update of $\hat{A}$ in Algorithm \ref{alg:re-q-12}, we have
    \begin{equation*}
        0\in\partial\lambda\|\hat{A}^{k+1}\|_L^1+\rho^k(\hat{A}^{k+1}-A^{k+1}W^{-1})+Y_1^{k}.
    \end{equation*}
    Note that $\rho^k(\hat{A}^{k+1}-A^{k+1}W^{-1})+Y_1^{k}=Y_1^{k+1}$, then according to Lemma \ref{lemma:diff-of-sigma}, we have
    \begin{equation*}
        -Y_1^{k+1}\in\lambda U\partial(\log(|\cdot|+\epsilon))|_{S}V^T,
    \end{equation*}
    where $USV^T=\hat{A}^{k+1}$ is SVD and 
    \begin{equation*}
        \partial(\log(|x|+\epsilon))=\left\{\begin{aligned}
            &\frac{1}{x+\epsilon}, &&x>0,\\
            &[-\frac{1}{\epsilon},\frac{1}{\epsilon}], &&x=0.
        \end{aligned}\right.
    \end{equation*}
    Therefore, we know
    \begin{equation*}
        \|Y_1^{k+1}\|_2\leq\lambda\|(S+\epsilon I)^{-1}\|_2\leq\frac{\lambda}{\epsilon}.
    \end{equation*}
    This indicates that the variable $Y^{(1)}$ is bounded. Similarly, the boundedness of variable $Y^{(2)}$ can be proved. On the other hand, notice that $\rho^k(\hat{A}^{k}-A^{k}W^{-1})=Y_1^{k}-Y_1^{k-1}$, we can obtain
    \begin{align*}
        &\langle Y_1^{k},\frac{Y_1^{k}-Y_1^{k-1}}{\rho^{k-1}}\rangle+\frac{\rho^k}{2}\|\hat{A}^{k}-A^{k}(\hat{W})^{-1}\|_F^2-\langle Y_1^{k-1},\frac{Y_1^{k}-Y_1^{k-1}}{\rho^{k-1}}\rangle-\frac{\rho^{k-1}}{2}\|\hat{A}^{k}-A^{k}(\hat{W})^{-1}\|_F^2\\
        =&\frac{\rho^k+\rho^{k-1}}{2(\rho^{k-1})^2}\|Y_1^{k}-Y_1^{k-1}\|_F^2.
    \end{align*}
    Similarly, we can prove
    \begin{align*}
        &\langle Y_2^{k},\frac{Y_2^{k}-Y_2^{k-1}}{\rho^{k-1}}\rangle+\frac{\rho^k}{2}\|\hat{B}^{k}-B^{k}\|_F^2-\langle Y_2^{k-1},\frac{Y_2^{k}-Y_2^{k-1}}{\rho^{k-1}}\rangle-\frac{\rho^{k-1}}{2}\|\hat{B}^{k}-B^{k}\|_F^2\\
        =&\frac{\rho^k+\rho^{k-1}}{2(\rho^{k-1})^2}\|Y_2^{k}-Y_2^{k-1}\|_F^2,\\
        &\langle Y_3^{k},\frac{Y_3^{k}-Y_3^{k-1}}{\rho^{k-1}}\rangle+\frac{\rho^k}{2}\|A^k(B^k)^T-X^{k}\|_F^2-\langle Y_3^{k-1},\frac{Y_3^{k}-Y_3^{k-1}}{\rho^{k-1}}\rangle-\frac{\rho^{k-1}}{2}\|A^k(B^k)^T-X^{k}\|_F^2\\
        =&\frac{\rho^k+\rho^{k-1}}{2(\rho^{k-1})^2}\|Y_3^{k}-Y_3^{k-1}\|_F^2.\\
        &\langle Y_4^{k},\frac{Y_4^{k}-Y_4^{k-1}}{\rho^{k-1}}\rangle+\frac{\rho^k}{2}\|X^{k}+S^k-M\|_F^2-\langle Y_4^{k-1},\frac{Y_4^{k}-Y_4^{k-1}}{\rho^{k-1}}\rangle-\frac{\rho^{k-1}}{2}\|X^{k}+S^k-M\|_F^2\\
        =&\frac{\rho^k+\rho^{k-1}}{2(\rho^{k-1})^2}\|Y_4^{k}-Y_4^{k-1}\|_F^2.
    \end{align*}
    Therefore,
    \begin{align*}
        &L_{\rho^{k}}(A^{k+1},B^{k+1},S^{k+1},X^{k+1},\hat{A}^{k+1},\hat{B}^{k+1};Y_i^k,1\leq i\leq 4)\\
        \leq&\mathcal{L}_{\rho^{k}}(A^{k},B^{k},S^{k},X^{k},\hat{A}^{k},\hat{B}^{k};Y_i^k,1\leq i\leq 4)\\
        =&\mathcal{L}_{\rho^{k}}(A^{k},B^{k},S^{k},X^{k},\hat{A}^{k},\hat{B}^{k};Y_i^{k-},1\leq i\leq 4)+\sum_{i=1}^4\frac{\rho^k+\rho^{k-1}}{2(\rho^{k-1})^2}\|Y_i^{k}-Y_i^{k-1}\|_F^2,
    \end{align*}
    noting that $(Y^{k},1\leq i\leq 3)$ is bounded and $(\rho^k+\rho^{k-1})/(2(\rho^{k-1})^2)$ is a geometric series, we know that $ L_{\rho^{k}}(A^{k+1},B^{k+1},S^{k+1},X^{k+1},\hat{A}^{k+1},\hat{B}^{k+1};Y_i^k,1\leq i\leq 4)$ is bounded. By continuously using the definition of $L$, i.e., \eqref{eq: model-l1-ALS-new}, we can successively obtain $\hat{A},\hat{B},A,B,X,S$ are bounded.
\end{proof}
Below we prove the convergence of Algorithm \ref{alg:re-q-12}.

\begin{theorem}
	Let $(A^k,B^k,S^k,X^k,\hat{A}^k,\hat{B}^k,Y_i^k, 1 \leq i \leq 4)$ be the iterative sequence generated by Algorithm \ref{alg:re-q-12}. If $\{Y_3^k\}$ is bounded, then $\{(A^k,B^k,S^k,X^k,\hat{A}^k,\hat{B}^k,Y_i^k, 1 \leq i \leq 4)\}$ is a Cauchy sequence, and its limit is the KKT point of \eqref{eq:rpca-decomposition-2}.
\end{theorem}

\begin{proof}
	Note that $\rho^k$ is a geometric series, we have
	\begin{align*}
		\sum_{k=0}^{\infty}\|\hat{A}^k - A^k(W)^{-1}\|_F &= \sum_{k=0}^{\infty}\frac{1}{\rho^{k}}\|Y_1^{k} - Y_1^{k-1}\|_F < \infty,
	\end{align*}
	which implies
	\begin{equation}\label{eq:0304-3}
		\lim_{k\rightarrow\infty}\|\hat{A}^k - A^k(W)^{-1}\|_F = 0,
	\end{equation}
	similarly,
	\begin{equation}\label{eq:0304-4}
		\lim_{k\rightarrow\infty}\|\hat{B}^k - B^k\|_F = 0, \lim_{k\rightarrow\infty}\|A^kB^k - X^k\|_F = 0,\ \lim_{k\rightarrow\infty}\|X^k+S^k-M\|_F = 0.
	\end{equation}
	On the other hand, we have
	\begin{align}\label{eq:0304-1}
		&\rho^{k}(A^{k+1}W^{-1} - \hat{A}^{k} - Y_1^kW^{-1}) \nonumber \\
        =& \rho^{k}(A^{k+1} - A^{k})W^{-2} + \rho^k(A^kW^{-1}-\hat{A}^k) - Y_1^kW^{-1},\nonumber\\
		=& \rho^{k}(A^{k+1} - A^{k})W^{-2} + \frac{\rho^{k}}{\rho^{k-1}}(Y_1^{k+1}W^{-1} - Y_1^k) - Y_1^kW^{-1},
	\end{align}
	and
	\begin{align}\label{eq:0304-2}
		&\rho^k(A^{k+1}B^{k} - X^{k})(B^k)^{T} - Y_3^k(B^k)^{T} \nonumber \\
		=& \rho^k(A^{k+1} - A^{k})B^{k}(B^{k})^{T} + \rho^k(A^{k}B^{k} - X^{k})(B^k)^{T} - Y_3^k(B^k)^{T} \nonumber \\
		=& \rho^k(A^{k+1} - A^{k})B^{k}(B^{k})^{T} + \frac{\rho^k}{\rho^{k-1}}(Y_3^k - Y_3^{k-1})(B^{k})^{T} - Y_3^k(B^k)^{T}.
	\end{align}
	By using the update rule of $A$, we know that the left-hand side of \eqref{eq:0304-1} and \eqref{eq:0304-2} are the gradients of the update function of $A$, which is 0 during the update. Therefore, the left-hand side of \eqref{eq:0304-1} and \eqref{eq:0304-2} sums to 0, yielding
	\begin{align*}
		\sum_{k=1}^{\infty}\|A^{k+1} - A^{k}\|_F \leq \sum_{k=1}^{\infty}\frac{1}{\rho^k}\phi < \infty,
	\end{align*}
	where
	\begin{align*}
		&\phi =\max\{\big(\mu(Y_1^{k+1}W^{-1} - Y_1^k)-Y_1^kW^{-1} + \mu(Y_3^{k} - Y_3^{k-1})(B^{k})^{T}\big)(W^{-2} + B^{k}(B^{k})^{T})^{-1}, k \geq 1\}< \infty.
	\end{align*}
	This proves that $A$ is a Cauchy sequence. Similarly, we can prove that $B$ is a Cauchy sequence. Furthermore, using \eqref{eq:0304-3} and \eqref{eq:0304-4}, we can prove that $\hat{A}$, $\hat{B}$, and $X,S$ are Cauchy sequences. Utilizing the lower semicontinuity property of the subdifferential of the Logarithmic norm, nuclear norm or Frobenius norm, we can show that the limit point $(A^{\infty},B^{\infty},S^{\infty},X^{\infty},\hat{A}^{\infty},\hat{B}^{\infty},Y_i^{\infty}, 1 \leq i \leq 4)$ satisfies
	\begin{align*}
		&0 \in \partial\lambda h_1(\hat{A}^{\infty})+Y_1^{\infty},\ 0 \in \partial\lambda h_2(\hat{B}^{\infty})+Y_2^{\infty},\nonumber\\
		&P_{\Omega^c}(Y_4^{\infty})=0,\ 0\in\partial\|P_{\Omega}(S^{\infty})\|_1+P_{\Omega}(Y_4^{\infty}),\nonumber\\
		&\hat{A}^{\infty} = A^{\infty}(W)^{-1},\ \hat{B}^{\infty} = B^{\infty},\ A^{\infty}(B^{\infty})^T = X^{\infty},\ X^{\infty}+S^{\infty}-M=0.
	\end{align*}
	This proves that the limit point of the iterative sequence is the KKT point of \eqref{eq: model-l1-ALS-new}.
\end{proof}

Theorem \ref{thm51} indicates that when the first $\text{rank}(X)$ elements of the weighted matrix $W$ are close to the corresponding singular values of matrix $X$, the expressions (1), (2), (3) and (4) in Theorem \ref{thm51} are close to $\text{rank}(X)$, representing that the weighted Schatten-q norm and Logarithmic norm are good relaxations of the rank function. Therefore, an important problem in weighted matrix low-rank recovery is how to determine a sufficiently good weight matrix $W$. Fortunately, \cite{shang2017bilinear} and \cite{chen2021logarithmic} have shown that when the weight matrix $W$ is the identity matrix, the corresponding matrix recovery model has good recovery performance. Therefore, when $X$ is the output of an unweighted matrix recovery model, i.e., when the weighted matrix $W^0$ is the identity matrix, the result of the recovery is $X$, which can be considered close to the original matrix to be recovered. In the proposed method for determining the new weight matrix, we consider the singular value decomposition of $X$ as $X=U\Sigma V^T$, and then define the singular values $\Sigma$ as the new weight matrix $W^{1}$, which leads to better matrix recovery. We summarize the above idea into Algorithm \ref{alg:reweighted-summery}.

\begin{algorithm}
\caption{Reweighted quasi norm matrix Recovery}\label{alg:reweighted-summery}
\begin{algorithmic}
\STATE {\textbf{Input.} Observed matrix $M\in\mathbb{R}^{n\times m}$, observation set $\Omega$, parameters $\lambda$, $\epsilon>0$, and rank $r$.}
\STATE {\textbf{Step 0.} Initialize $W^0=I_{r}$. Iteratively perform the following steps for $k=0,1,\dots$:}
\STATE {\textbf{Step 1.} Call the algorithm \ref{alg:re-q-12} for matrix recovery with weight matrix $W^k$, obtaining the output $X^k$.}
\STATE {\textbf{Step 2.} Perform singular value decomposition of $X^k$ to obtain the singular value matrix $\Sigma^k$. Update
\begin{equation*}
W^{k+1}=\max\left\{\Sigma^{k},\epsilon\right\},
\end{equation*}
and update $k\rightarrow k+1$.}
\STATE {\textbf{Output.} $X^{k}$.}
\end{algorithmic}
\end{algorithm}

The numerical experiments will be presented in future work.

\clearpage

\bibliographystyle{unsrt}   
\bibliography{ref} 

\clearpage

\end{document}